\documentclass[10pt,a4paper]{article}
\usepackage{amsthm}
\usepackage{amsmath}
\usepackage{amsfonts}
\usepackage{amssymb}
\usepackage{graphicx}%
\usepackage[numbers]{natbib}

\newtheorem{theorem}{Theorem}

\newtheorem{definition}[theorem]{Definition}

\newtheorem{lemma}[theorem]{Lemma}

\newtheorem{proposition}[theorem]{Proposition}
\newtheorem{remark}[theorem]{Remark}

\newcommand{\cpfadd}[1]{{\color{blue}#1}}

\usepackage[hyperindex = true, colorlinks = true, linktocpage = true, linkcolor=blue, citecolor=blue, bookmarks, pdftitle={Global existence, regularity and a probabilistic scheme for a class of ultraparabolic Cauchy problems}, pdfauthor={Christian Fries}]{hyperref}

\renewenvironment{proof}[1][Proof]{\noindent\textbf{#1.} }{\ \rule{0.5em}{0.5em}}

\begin{document}

\title{Global existence, regularity and a probabilistic scheme for a class of ultraparabolic Cauchy problems
}

\author{Christian Fries$^1$, J\"org Kampen$^{2}$}

\maketitle

\begin{abstract}

In this paper we establish a constructive method in order to show global existence and regularity for a class of degenerate parabolic Cauchy problems which satisfy a weak H\"{o}rmander condition on a subset of the domain where the data are measurable and which have regular data on the complementary set of the domain. This result has practical incentives related to the computation of Greeks in reduced LIBOR market models, which are standard computable approximations of the HJM-description of interest rate markets. The method leads to a probabilistic scheme for the computation of the value function and its sensitivities based on Malliavin calculus. From a practical perspective the main contribution of the paper is an Monte-Carlo algorithm which includes weight corrections for paths which move in time into a region where a (weak) H\"{o}rmander condition holds.

\\[0.6ex] {\it 2000 AMS subject
classification:} 60H10, 62G07, 65C05 
\end{abstract}

\newpage


\tableofcontents

\newpage

\section{Introduction}
In mathematical finance, e.g., derivative valuation and risk management, stochastic differential equations (SDEs) are used to model a possible high dimensional state space which depends on several sources of risk (diffusions), also known as \textit{factors}. Practical incentives lead to a reduction of the number of factors. Reduced LIBOR market models are a prominent example. Such a reduction may be admissible when the payoff (or sensitivity) depends only very weakly (smoothly) on the corresponding factor. In the equivalent formulation using partial differential equations (PDEs), this leads to class of ultra-parabolic Cauchy problems, where diffusions  degenerate in a strict sense on parts of the domain where the initial data are smooth. Here, 'strict sense' means that the diffusion equation my have no density. Consequently even a weak H\"{o}rmander condition may not hold on the whole domain. Here, the introduction of a weak H\"{o}rmander condition in this paper is related to the fact that diffusion coefficients in financial models may not be $C^{\infty}$-functions but only Lipschitz globally$^{3}$\footnotetext[3]{We thank an associate editor of F\&S for this remark commenting a former version of this paper; furthermore we adopted the term 'ultraparabolic' from a referee of the former version which seems to fit better than the term 'semi-elliptic' which was used in the former version.}. 


For such classes of ultra-parabolic Cauchy problems the computation of Greeks seems to be difficult. In general, the standard methods of Malliavin calculus fail because the Malliavin-covariance matrix is not invertible in any $L^p$-sense for $p\geq 1$ as required. Nevertheless volatility matrices $\sigma$ of the models used in practice are Lipschitz-continuous functions and satisfy a weak ellipticity, i.e., they satisfy
\begin{equation}
	\sigma \sigma^T \ \geq \ 0 \text{.}
\end{equation}
If the volatility matrices satisfy a linear growth, and the payoffs and the data satisfy a certain regularity condition, then Peano's method adapted to stochastic ODEs is the best method available in order to establish global existence and (weak) regularity.   
This standard theory of stochastic differential equations provides Feller-continuous value functions which solve the associated Cauchy problems. If the payoffs are only continuous (but nonnegative) then the standard theory provides only lower semi-continuous solutions. For the computation of sensitivities (Greeks in finance) it is desirable to have more regular solutions. Moreover, we note that in the situation of financial applications the restriction to bounded continuous payoffs (or, equivalently, initial data) is  a limitation. We would like to have at least Lipschitz continuous payoffs where the growth of the payoff has an exponential bound (note that Cauchy problems in finance are formulated in logarithmic coordinates). Furthermore, we have observed that there are some limitations concerning the regularity of the initial data. In order to obtain progress in this direction we shall impose partial regularity. We assume measurable data on a domain where the Malliavin-covariance matrix is invertible and regular data elsewhere. The reason is quite obvious: on a certain subspace the degenerate operator of the factor-reduced problem operates similar as a vector-field, and a vector-field merely transports irregularities.

A standard theorem concerning ordinary stochastic differential equations (for statement and proof cf. \cite{O}) is the following.
\begin{theorem}
	\label{sdethm}
	Let $T>0$ and let $b:[0,T]\times {\mathbb R}^n\rightarrow {\mathbb R}^n$, and $\sigma :[0,T]\times {\mathbb R}^n\rightarrow {\mathbb R}^{n\times m}$
	be measurable functions, where
	\begin{equation}
		|b(t,x)|+|\sigma(t,x)|\leq C(t+|x|);~x\in {\mathbb R}^n,~t\in [0,T] \text{,}
	\end{equation}
	for some constant generic $C>0$ and with $|\sigma(t,x)|=\sqrt{\sum_{ij}|\sigma_{ij}|^2}$ ($|.|$ denoting the Euclidean norm), and such that
	\begin{equation}
		\label{Lip}
		|b(t,x)-b(t,y)|+|\sigma(t,x)-\sigma(t,y)|\leq C|x-y|;~x\in {\mathbb R}^n,~t\in [0,T].
	\end{equation}
	Let $Z$ be a random variable independent of the $\sigma$-algebra ${\cal F}_{\infty}$ generated by $W(s),~s\geq 0$ and such that $E(|Z|^2)<\infty$.
	Then the stochastic differential equation 
	\begin{equation}\label{sdebasic}
		dX(t)=b(t,X(t))dt+\sigma(t,X(t))dW(t),~0\leq t\leq T,~X(0)=Z
	\end{equation}
	has a unique $t$-continuous solution $(t,\omega)\rightarrow X(t,\omega)$, where each component of $X(t,\omega)$ belongs to the space
	$$
	{\cal V}(0,T):=\left\lbrace h(t,\omega):[0,\infty)\times \Omega\rightarrow {\mathbb R}|h~\mbox{satisfies (i),(ii), (iii)}\right\rbrace ,
	$$ 
	along with the conditions
	\begin{itemize}
		\item[(i)]$~(t,\omega)\rightarrow h(t,\omega)$ is ${\cal B}\times {\cal F}$-measurable, where ${\cal B}$ denotes the Borel $\sigma$-algebra on $[0,\infty)$,
\item[(ii)] $h(t,\omega)$ is ${\cal F}_t$-adapted,
\item[(iii)] $E\left[\int_0^T h(t,\omega)^2dt\right]<\infty$. 
\end{itemize}
\end{theorem}

\cpfadd{}

Note that no strict ellipticity condition is involved. Indeed, even a weak ellipticity condition such as the H\"{o}rmander condition is not involved. Since we are interested in Greeks, our main concern is the regularity of the functions
\begin{equation}\label{expect}
	(t,x)\rightarrow u(t,x) \ = \ E^x\left( f\left( X_t\right) \right) ,
\end{equation}
and its derivatives (with respect to some arguments or with respect to other parameters). Here, $X_t$ is the solution of (\ref{sdebasic}), (which will satisfy $X(0)=x$ in general), and $f$ is some function. The standard theory derives a global existence result for $u$ as a solution for the associated Cauchy problem
\begin{equation}\label{CP}
\left\lbrace \begin{array}{ll}
 \frac{\partial u}{\partial t}-\mbox{Tr}\left( \sigma\sigma^TD^2u\right)-\sum_{i=1}^nb_i\frac{\partial u}{\partial x_i}=0,\\
 \\
 u(0,x)=f(x), 
\end{array}\right.
\end{equation}
via the It\^{o}-formula for data $f\in C^2\left({\mathbb R}^n\right)_K$, i.e., for data which are twice differentiable and have a compact support $K$.
The closure of such a function space of data is the space $C_0\left({\mathbb R}^n\right) $, i.e., the space of continuous functions which vanish at infinity (note that this type of closure is valid for any locally compact Hausdorff space). 
The best result that we could obtain for data $C_0\left({\mathbb R}^n\right)$ 
using a standard argument is that the function $u$ in (\ref{expect}) is continuous. However, it is not clear in which sense it might be a solution of the associated Cauchy problem (\ref{CP}).
It seems impossible to verify that such a limit is a solution in the viscosity sense. It may be proved that the limit is a solution in a very weak sense. This type of result is not sufficient for computing sensitivities of ultraparabolic models, a situation which may arise for reduced standard models in interest rate markets. In such an application, we want to compute Greeks and this implies that we want to compute first or second derivatives of value functions, and since we want to produce numbers, it is not sufficient to have existence in a distributional sense. Second, we want to allow for data which may have exponential growth at infinity- this is the nature of standard payoffs rewritten in coordinates $x_i=\ln (S_i)$ where $S_i$ are the lognormal coordinates. Now the standard theory for Greeks is the Malliavin calculus which was first developed as a probabilistic reformulation of H\"{o}rmander's result in \cite{H}. However, this calculus cannot be applied directly in the context of the class of ultra-parabolic equations which is considered here. The reason is the following. Malliavin calculus defines natural closure spaces $D^{r,p}$ for differentiability with respect to random increments of order $r$ in $L^p$-sense (with probabilistic interpretation) for random vectors $X=(X^1,\cdots ,X^m)$. A cornerstone of the theory is the so-called Malliavin covariance matrix which is defined by
\begin{equation}\label{covar}
\gamma^{ij}_X=\sum_{r=1}^n\int_0^{\infty}D^r_tX^i_tD^r_tX^j_t dt.
\end{equation}
The computation of Greeks is then based on integration by parts formulae which are typically of the  form
\begin{equation}
E\left(\frac{\partial}{\partial x_i}f(X)Y \right)=E(f(X)W^i(X,Y)), 
\end{equation}
where $X,Y\in D^{1,2}$ are some random vectors and $W^i$ is a weight functional involving the inverse of the covariance matrix (\ref{covar}) as a factor. This covariance matrix has to be not only invertible a.s., it has to satisfy a certain kind of $L^p$-invertability, i.e., it has to satisfy
\begin{equation}\label{densitycond}
E\left( \left( \mbox{det}\left( \gamma^{ij}_X\right)^{-1}\right)  \right)<\infty~\mbox{ for all}~p>1. 
\end{equation}
However, this condition (which is related to the existence of densities) is not satisfied by reduced financial market models in general (even if the disperison coefficients are smooth). Therefore, we need an extension of the theory to some classes of ultraparabolic equations which subsume reduced financial market models where (\ref{densitycond}) does not hold. 
Finally, for smooth bounded data and smooth coefficient functions with bounded derivatives it may be possible to prove regularity of the function
\begin{equation}\label{exfunc}
x\rightarrow E^x_{S}\left(f\left(X_{t\wedge\tau}\right) \right) 
\end{equation}
for some $S\subset {{\mathbb R}^n} $ where $t\wedge \tau$ is the minimum of $t$ and the first exit time of $\Omega$. If $t\wedge \tau$ is a well-defined stopping time, and the data $f$ and coefficient functions $\left( \sigma\sigma^T\right)_{ij}$ are as indicated, then the function
(\ref{exfunc}) may be smooth. This may be proved by differentiation of the processes $X^x_t$ with respect to the starting point $x$ where you may prove existence for the derivatives using standard techniques for stochastic ODEs (Picard iterations in appropriate functions spaces). We have not found this in the literature, and it is not a main purpose of this paper. Therefore we introduce this as an assumption saying that
\begin{itemize}
\item[(A)] is satisfied if the function (\ref{exfunc}) is $C^{\infty}$ for all $t\geq 0$ and $S={\mathbb R}^n\setminus \Omega$ and for some domain $\Omega\subset H$, where $H\subset {\mathbb R}^n$ is the set where the H\"{o}rmander condition associated with the process $X_t$ holds (for a definition of the H\"{o}rmander condition cf. below). 
\end{itemize}

In the next section we consider a class of ultraparabolic Cauchy problems which are defined in terms of smooth vector fields on a subdomain where a classical H\"{o}rmander condition holds. This result may be applied in the case of a classical reduced LIBOR market model. However, in the context of stochastic volatility extensions of the LIBOR market model results for Lipschitz-continuous coefficients are desirable. This requires certain weak H\"{o}rmander conditions which are considered in section 3 of this paper. In this case there are some restrictions with respect to regularity as may be expected from the perspective of Malliavin calculus. In section 4  of this paper we construct a weighted MC-algorithm for the class of ultraparabolic models considered in this paper.

\section{Global regularity for a class of degenerate para\-bolic equations}
\label{sec:finstochhypo:globalExistenceClassOfDegenerateParabolicEquations}
The standard existence theory of SDEs which leads to theorem \ref{sdethm} is essentially a generalisation of Picard's iteration methods of ODEs in the context of infinite dimensional state spaces. This method together with the Feynman-Kac formalism may be applied to get regularity results if (derivatives of ) the data and (derivatives of) the coefficicients are smooth and of some polynomial decay at spatial infinity, i.e., at least for data $f\in \cap_{s\in {\mathbb R}}H^s$ and coefficients $\sigma\sigma^T=\left( \sigma\sigma^T_{ij} \right)$ with $\sigma\sigma^T_{ij}\in \cap_{s\in {\mathbb R}}H^s$, where $H^s\equiv H^s\left({\mathbb R}^n\right)$ are the Sobolev spaces of exponent $s\in {\mathbb R}$. In this case we may differentiate the expectation value expression of the value function and apply the standard method again to gain more regularity. However this method cannot be applied if the data or the coefficient functions of the SDE are of lower regularity. Both features are typical for models of mathematical finance. In addition we typically have exponential growth of data at infinity for such models. In this section we weaken the regularity condition and the growth condition of the data $f$. Indeed we shall allow for exponential growth of the data, and we allow for data which are measurable on a part of the domain where a classical H\"{o}rmander condition holds. We speak of a partial H\"{o}rmander condition if the H\"{o}rmander condition does not hold on the whole domain of the Cauchy problem. In order to formulate a classical H\"{o}rmander condition we need smoothness of the coefficients. In many situation such as reduced versions of classical Libor market models these conditions are satisfied. However, for some stochastic volatility models we may have weaker regularity of the coefficients too, and in such cases we need a weaker form of the partial H\"{o}rmander condition. This extension will be considered in the next section. In this section assuming regular coefficients we may formulate the partial H\"{o}rmander condition classically in terms of vector fields. First we may reformulate the Cauchy problem in (\ref{CP}) in terms of vector fields as follows. 
Consider a matrix-valued function $x\rightarrow (v_{ji})^{n,m}(x),~1\leq j\leq n,~0\leq i\leq m$ on ${\mathbb R}^n$, and $m$ vector fields 
\begin{equation}
V_i=\sum_{j=1}^n v_{ji}(x)\frac{\partial}{\partial x_j},
\end{equation}
where $0\leq i\leq m$. 
Consider the 
Cauchy problem on $[0,\infty)\times {\mathbb R}^n$
\begin{equation}
	\label{projectiveHoermanderSystemgenx2}
	\left\lbrace \begin{array}{ll}
		\frac{\partial u}{\partial t}=\frac{1}{2}\sum_{i=1}^mV_i^2u+V_0u,\\
		\\
		u(0,x)=f(x).
	\end{array}\right.
\end{equation}
In this section we consider smooth vector fields, i.e., $v_{ji}\in C^{\infty}\left(\left[0,\infty \right)\times {\mathbb R}^n \right)$. 
Define for all $x\in {\mathbb R}^n$
\begin{equation}\label{Hoergenx}
\begin{array}{ll}
H_x:=\mbox{span}{\Big\{} &V_i(x), \left[V_j,V_k \right](x), \\
\\
&\left[ \left[V_j,V_k \right], V_l\right](x),\cdots |1\leq i\leq m,~0\leq j,k,l\cdots \leq m {\Big \}}.
\end{array}
\end{equation}
Let
\begin{equation}
H:=\left\lbrace x\in {\mathbb R}^n|H_x={\mathbb R}^n\right\rbrace 
\end{equation}
be the set of points where the H\"{o}rmander condition holds. A subspace of $H_x$ which is induced by a set of vectors $A_1(x),\cdots A_p(x)$ will be denoted by $H_x\left[A_1,\cdots,A_p \right]$. It is clear how this classically formulated condition may be reformulated in the context of SDEs related to diffusion processes $X_t$. In this context we speak of the H\"{o}rmander condition related to the process $X_t$ (you may find this in \cite{KS}). 
\begin{remark}\label{degrem}
In the followin as usual we say that $f:{\mathbb R}^n\rightarrow {\mathbb R}$ is $C^{\infty}$ at $x$ if partial derivatives of arbitrary order exist at $x\in {\mathbb R}^n$. For the sake of simplicity we state the theorem in the case where $\Omega\subseteq {\mathbb R}^n$ is a domain in ${\mathbb R}^n$. In finance there may be situations where points of lower regularity of the datat may be located on a lower-dimensional manifold. An extended version of the following theorem may be stated  using the word 'domain relative to some subspace of ${\mathbb R}^n$'  for an set $\Omega\subseteq {\mathbb R}^m,~m\leq n$, which is open with respect to the relative topology of ${\mathbb R}^m$, and such that $\partial \Omega=\partial \left({\mathbb R}^m\setminus \overline{\Omega} \right)$.  Here $\overline{\Omega}$ denotes the closure and $\partial \Omega$ denotes the boundary of $\Omega$. Note that a domain relative to some subspace of ${\mathbb R}^n$ is not necessarily connected (this is also true for usual domains). Furthermore we say that a parabolic operator of the form (\ref{CP}) degenerates in the complement of a domain $\Omega$ relative to some subspace of ${\mathbb R}^n$ if $\sigma\sigma^T\equiv 0$ on ${\mathbb R}^n\setminus \Omega$. Analogously for the reformulations of the equation (\ref{CP}) in (\ref{projectiveHoermanderSystemgenx2}). 
\end{remark}

Next we state
 
\begin{theorem}
	\label{mainthm}
	Let $1\leq p\leq \infty$ and let $v_{ji}\in C^{\infty}\left(\left[0,\infty \right)\times {\mathbb R}^n \right)$ for $1\leq i\leq m$ and $1\leq j\leq n$. Assume that the set $\Omega \subseteq H\subseteq {\mathbb R}^n$, where $\Omega \subseteq {\mathbb R}^m$ is a domain, and $H$ is the set of points in ${\mathbb R}^n$, where the H\"{o}rmander condition holds.  Assume either that (A) is satisfied or the parabolic operator in (\ref{CP}) degenerates in the complement of a domain $\Omega$ (cf. remark \ref{degrem}), and that the initial data function $f:{\mathbb R}^n\rightarrow {\mathbb R}$ satisfies
	\begin{equation}\label{payoff}
	\begin{array}{ll}
		(i) & ~x\rightarrow f(x)\exp(-C|x|)\in L^p\left({\mathbb R}^n\right) \mbox{ for some $C>0$},\\
		\\
		(ii) & \Omega\cup \left\lbrace x|~f~\mbox{is}~C^{\infty}~\mbox{at}~x\right\rbrace ={\mathbb R}^n.
	\end{array}
\end{equation}
Then the Cauchy problem (\ref{projectiveHoermanderSystemgenx2})
	on $[0,\infty)\times {\mathbb R}^n$ has a global classical solution $u$, where
\begin{equation}
u\in C^{\infty}\left(\left(0,\infty\right) \times {\mathbb R}^n \right). 
\end{equation}

\end{theorem}

\begin{proof}
We povide the proof in the case where $H$ is a domain in ${\mathbb R}^n$ and where the operator degenerates in the complement of $H$. This is the case which most often occurs in practice. The proof can easily extended to the case where $m\leq n$, or an additional assumption (A) is satsified.
First we observe that it is sufficient to prove the theorem under the stronger assumption of a payoff $f\in C_0$ as discussed in the introduction. First note that we can transform the original Cauchy problem for $u$ to a problem for
\begin{equation}\label{transa}
	\tilde{u}:=e^{-d(x)}u:=e^{-\sqrt{a+q|x|^2}}u
\end{equation}
for some $a>0$, $q>C^2$, and where $|.|$ denotes the Euclidean norm. The Cauchy problem (\ref{projectiveHoermanderSystemgenx2}) is equivalent to a problem of the form
\begin{equation}
	\label{projectiveHoermanderSystemgenx2eq}
	\left\lbrace \begin{array}{ll}
		\frac{\partial u}{\partial t}=\frac{1}{2}\sum_{i,j=1}^n a_{ij}\frac{\partial^2}{\partial x_i\partial x_j}u+\sum_{i=1}^nb_i\frac{\partial u}{\partial x_i},\\
		\\
		u(0,x)=f(x),
	\end{array}\right.
\end{equation}
where $\left( a_{ij}\right)=\sigma\sigma^T\geq 0$
Then $\tilde{u}$ solves an equivalent problem with identical diffusion term but transformed drift vector $\tilde{{\bf b}}:={\bf b}-\frac{1}{2}\nabla d\cdot \sigma\sigma^T$ and an additional potential term $\tilde{c}:=c+{\bf b}\cdot \nabla d-\frac{1}{2}\mbox{tr}\left(\sigma\sigma^T\right)D^2 d-\frac{1}{2}|\nabla d\sigma|^2$. Here $D^2d$ denotes the Hessian of the function $d$ and $\mbox{tr}$ denotes the trace of a matrix. This amounts to a shift of the drift and an additional potential term $c$. The latter is not decisive for the result nor is the shift of the drift. Note that the set where the H\"{o}rmander condition holds may be altered although the difference can only due to the commutators where the drift is involved. We denote the set where the H\"{o}rmander condition holds for the transformed equation by $\tilde{H}$. We observe
\begin{lemma}
There is a transformation close to the transformation (\ref{transa}) such that
\begin{equation}
H= \tilde{H}.
\end{equation}
\end{lemma}
\begin{proof}
We have
\begin{equation}
 \frac{1}{2}\nabla d\cdot \sigma\sigma^T(x) \in H_x\left[V_1,\cdots ,V_m\right]=H_x\left[\sigma_1,\cdots ,\sigma_n\right]
\end{equation}
for each $x\in H$.
\end{proof}

Then we have
\begin{lemma}\label{translemma}
It suffices to show that there exists a solution $\tilde{u}$  of the Cauchy problem
\begin{equation}
	\label{projectiveHoermanderSystemgenx2eqsimpl}
	\left\lbrace \begin{array}{ll}
		\frac{\partial \tilde{u}}{\partial t}=\frac{1}{2}\sum_{i,j=1}^n a_{ij}\frac{\partial^2}{\partial x_i\partial x_j}\tilde{u}+\sum_{i=1}^n\tilde{b}_i\frac{\partial \tilde{u}}{\partial x_i},\\
		\\
		\tilde{u}(0,x)=\tilde{f}(x),
	\end{array}\right.
\end{equation}
on $[0,\infty)\times {\mathbb R}^n$, and where the payoff $\tilde{f}$ satisfies
\begin{equation}\label{payoff2}
	\begin{array}{ll}
		(i) & ~x\rightarrow \tilde{f}(x)\in L^p\left({\mathbb R}^n\right) \mbox{ for some $C>0$},\\
		\\
		(ii) & \tilde{H}\cup \left\lbrace x|~\tilde{f}~\mbox{is}~C^{\infty}~\mbox{at}~x\right\rbrace ={\mathbb R}^n.
	\end{array}
\end{equation}
 has a global classical solution $u$, where
\begin{equation}
\tilde{u}\in C^{\infty}\left(\left(0,\infty\right) \times {\mathbb R}^n \right). 
\end{equation}
\end{lemma}
\begin{proof}
If $\overline{X}$ solves
\begin{equation}
d\overline{X}_t=\tilde {\bf b}\left(\overline{X}_t\right) dt+\sigma \left(\overline{X}_t\right)dW_t,~\overline{X}_0=x.
\end{equation}
then the Feynman-Kac formula tells us that it suffices to prove the regularity of the function
\begin{equation}\label{exp}
(t,x)\rightarrow 
E^x\left(\exp\left(-\int_0^tc\left(\overline{X}_s\right)ds \right) f\left(\overline{X}_t\right) \right) 
\end{equation}
Then we may differentiate with respect to $x$ using the product rule, the  chain rule and compute 
the equation for the matrix-valued process $\frac{\partial}{\partial x_j}\overline {X}=Y^{x}_{ij}$ to be
\begin{equation}
\frac{\partial}{\partial x_j}\overline {X}=Y^{x}_{ij}=\delta_{ij}+\sum_{k=1}^n\int_0^t\frac{\partial}{\partial x_k}\left(\tilde{b_i}\right) Y^{k}_j(s)ds+\sum_{l,k=1}^n\int_0^t\frac{\partial}{\partial x_k}\sigma_{il}(X_s)Y^k_j(s)dW^l_s.
\end{equation}
Since a strong solution $Y^x_i$ of the latter SDE exists, the first derivative of (\ref{exp}) exists. Similarly for higher derivatives.  
\end{proof}

The following argument then proves the statement of \eqref{translemma}.
\begin{remark}
	Note that we may choose $q>C^2>0$ such that the initial data decay exponentially as $|x|\uparrow \infty$. From now on we assume that $q$ is chosen in this way.
\end{remark}  

\subsection{Existence of the Vector Field}

First we have
\begin{proposition}\label{prop}
 Assume that $\mu_i\in C^1_b\left([0,\infty)\times {\mathbb R}^n \right) $ and let $g\in C^1_b\left([0,\infty) \times {\mathbb R}^n \right) $. Then there exists a smooth global flow ${\cal F}^t$ generated by the vector field 
\begin{equation}\label{vec1*}
\sum_{i=1}^n\mu_i(x)\frac{\partial}{\partial x_i}
\end{equation}
 on $[0,\infty)\times {\mathbb R}^{n}$ such that the first order equation problem
\begin{equation}\label{vecf}
\begin{array}{ll}
\frac{\partial u}{\partial t}=\sum_{i=1}^n\mu_i(x^{n})\frac{\partial}{\partial x_i}u+g(x^{n}),\\
\\
~~u(0,x^{n})=f(x^{n}),
\end{array}
\end{equation}
has the solution
\begin{equation}\label{sol}
u(t,x^{n})=f\left({\cal F}^t x^{n}\right)+\int_0^tg({\cal F}^{t-s}x^{n})ds. 
\end{equation}

\end{proposition}

\begin{proof}
Consider the characteristic form
\begin{equation}
\chi_L(z,\xi)=\xi_0-\sum_{i=1}^n\mu_i\xi_i
\end{equation}
of the operator $L\equiv\frac{\partial}{\partial t}-\sum_{i=1}^n\mu_i \frac{\partial}{\partial x_i}$, where $\xi = (\xi_{0},\xi_{1},\ldots,\xi_{n})$. The surface $S:=\left\lbrace t=0 \right\rbrace$ has a constant normal vector $(1,0,\cdots ,0)$, hence is non-characteristic for the surface $S$, i.e. at any point $z=(t,x)$ we have 
\begin{equation}
(1,0,\cdots ,0)\notin \mbox{char}_z(L):=\left\lbrace \xi\neq 0|\xi_0-\sum_{i=1}^n \mu_i\xi_i=0 \right\rbrace \text{.}
\end{equation}
Hence, basic PDE-theory tells us that the first order Cauchy problem has a unique local solution in a sufficiently small neighborhood of the surface $S$ and is given in the form of solutions of associated ODEs along its characteristic curves. This leads to a solution up to a time $T_1$. Next we may iterate the argument in time. Assume that this does not lead to a global solution but to a limit $T_{\infty}>0$. Then on the time horizon $\left[0,T_{\infty}\right]$ we have a classical solution. Moreover the solution has a representation on this horizon as a family of ODE-solutions along characteristic curves, and where the assumptions on the coefficients	  imply that this family of solutions is uniformly bounded up to time $T_{\infty}$. Hence we may apply the first order PDE argument above again for the Cauchy problem with initial data $S_{T_{\infty}}:=\left\lbrace t=T_{\infty} \right\rbrace$ and extend the solution beyond the horizon $\left[0,T_{\infty}\right]$.     Hence there is a unique global solution. 
For each $x_0^{n}\in {\mathbb R}^{n}$ the flow ${\cal F}_t$ of the vector field $\sum_i \mu_i\frac{\partial}{\partial x_i}$ defines a characteristic curve $x_0^{n}(t):={\cal F}_tx_0^{n}$. Note that
\begin{equation}
{\cal F}^t x^{n} 
\end{equation}
is a solution of the homogeneous Cauchy problem
\begin{equation}\label{vecfhom}
\begin{array}{ll}
\frac{\partial u}{\partial t}=\sum_{i=1}^n\mu_i(x)\frac{\partial}{\partial x_i}u,\\
\\
~~u(0,x^n)=f(x^{n}),
\end{array}
\end{equation} 
and then the form of the solution (\ref{sol}) of the inhomogenous equation follows from Duhamel's principle.
\end{proof}

\subsection{Construction of the solution via an AD-Scheme}

Next we note that on the relative domain $H$ the Kusuoka-Stroock estimates hold. We have

\begin{theorem}
Let the assumption of (\ref{mainthm}) be satisfied and let $T>0$. 
Then the law of the diffusion process $X$ exists on a domain $\Omega\subseteq {\mathbb R}^n$ is absolutely continuous with respect to the Lebesgue measure, and the density $p$ exists and is smooth, i.e., on a domain $\Omega\subseteq {\mathbb R}^n$ we have
\begin{equation}
\begin{array}{ll}
p:(0,T]\times \Omega\times \Omega\rightarrow {\mathbb R}\in C^{\infty}\left( (0,T]\times \Omega\times \Omega\right). 
\end{array}
\end{equation}
Moreover, for each nonnegative natural number $j$, and multi-indices $\alpha,\beta$ there are increasing functions of time
\begin{equation}\label{constAB}
A_{j,\alpha,\beta}, B_{j,\alpha,\beta}:[0,T]\rightarrow {\mathbb R},
\end{equation}
and functions
\begin{equation}\label{constmn}
n_{j,\alpha,\beta}, 
m_{j,\alpha,\beta}:
{\mathbb N}\times {\mathbb N}^d\times {\mathbb N}^d\rightarrow {\mathbb N},
\end{equation}
such that 
\begin{equation}\label{pxest}
{\Bigg |}\frac{\partial^j}{\partial t^j} \frac{\partial^{|\alpha|}}{\partial x^{\alpha}} \frac{\partial^{|\beta|}}{\partial y^{\beta}}p(t,x,y){\Bigg |}\leq \frac{A_{j,\alpha,\beta}(t)(1+x)^{m_{j,\alpha,\beta}}}{t^{n_{j,\alpha,\beta}}}\exp\left(-B_{j,\alpha,\beta}(t)\frac{(x-y)^2}{t}\right) 
\end{equation}
Moreover, all functions (\ref{constAB}) and  (\ref{constmn}) depend on the level of iteration of Lie-bracket iteration at which the H\"{o}rmander condition becomes true.
\end{theorem}
For a proof consider \cite{KS}.

Next we define a scheme for the constructive solution of the Cauchy problem. It is convenient to define the scheme time-step by time-step on domains $[l-1,l]\times {\mathbb R}^n,~l\geq 1$. We introduce the time transformation
\begin{equation}
t=\rho \tau,
\end{equation}
where $\rho>0$ will be a small number. The transformed solution to the Cauchy problem
\begin{equation}
u^{\rho}(\tau,x)=u(t,x)
\end{equation}
satisfies
\begin{equation}
	\label{projectiveHoermanderSystemgenx2rho}
	\left\lbrace \begin{array}{ll}
		\frac{\partial u^{\rho}}{\partial \tau}=\rho\frac{1}{2}\sum_{i=1}^mV_i^2u^{\rho}+\rho V_0u^{\rho}\\
		\\
		u^{\rho}(0,x)=f(x),
	\end{array}\right.
\end{equation}
or, equivalently,
\begin{equation}
	\label{projectiveHoermanderSystemgenx2eqrho}
	\left\lbrace \begin{array}{ll}
		\frac{\partial u^{\rho}}{\partial \tau}=\rho\frac{1}{2}\sum_{i,j=1}^n a_{ij}\frac{\partial^2}{\partial x_i\partial x_j}u^{\rho}+\rho \sum_{i=1}^nb_i\frac{\partial u^{\rho}}{\partial x_i},\\
		\\
		u^{\rho}(0,x)=f(x).
	\end{array}\right.
\end{equation}
Note that the vector field
\begin{equation}\label{tildev}
\tilde{V}_0:=\sum_{i=1}^nb_i\frac{\partial}{\partial x_i}
\end{equation}
is not identical with $V_0$ in general and is denoted by (\ref{tildev}) henceforth.
The restriction of $u^{\rho}$ to the domain $[l-1,l]\times {\mathbb R}^n$ is denoted by $u^{\rho,l}$.
The solution is described by an iterative scheme for $u^{\rho,k,l}$ such that the solution has the form
\begin{equation}\label{funcseries}
u^{\rho ,l}(\tau,x)=u^{\rho,0,l}(\tau ,x)+\sum_{k\geq 1}\delta u^{\rho ,k,l}(\tau ,x),
\end{equation} 
where $\delta u^{\rho,k,l}=u^{\rho,l}-u^{\rho,k-1,l}$ and $k\geq 1$ is the iteration index at each time step $l\geq 1$. The representation in (\ref{funcseries}) is useful for proving convergence. For a small time step size $\rho>0$ we show at each time step $l\geq 1$
\begin{equation}
|\delta u^{\rho,k,l}|_{1,2}\leq c|\delta u^{\rho,k-1,l}|_{1,2}
\end{equation}
for some $c<1$, and then we use the semi-group property of the operator. Next we define the scheme at each time step $l\geq 1$. Let $(H_j)_{j\in J}$ denote the connected components domains of the domain $H$ (themselves each a domain), and let $\left( B^{\epsilon}_i\right)_{i\in I}$ be an open covering of ${\mathbb R}^n\setminus H$. We may assume that the both coverings are locally finite. At each time step we assume that the data $u^{\rho,l-1}(l-1,.)$ are given, where for $l=1$ we define $u^{\rho,0}(0,.)=f(.)$. For each $j\in J$ we construct
\begin{equation}
u^{\rho,0,l}_{H_j}:[l-1,l]\times H_j\rightarrow {\mathbb R}
\end{equation}
as a solution of a problem on $[l-1,l]\times H_j$, which is 
\begin{equation}
	\label{projectiveHoermanderSystemgenx2rhoH_j}
	\left\lbrace \begin{array}{ll}
		\frac{\partial u^{\rho ,0,l}_{H_j}}{\partial \tau}=\rho\frac{1}{2}\sum_{i=1}^mV_i^2u^{\rho ,0,l}_{H_j}+\rho V_0u^{\rho ,0,l}_{H_j},\\
		\\
		u^{\rho ,0,l}_{H_j}(l-1,x)=\chi_{H_j}(x)u^{\rho,l-1}(l-1,x) \mbox{ for }x\in H_j,
	\end{array}\right.
\end{equation}
where
\begin{equation}
\chi_{H_j}(x):=\left\lbrace \begin{array}{ll}
1 \mbox{ if }x\in H_j\\
\\
0 \mbox{ if }x\not\in H_j
\end{array}\right.
\end{equation}
denotes the characteristic function of $H_j$. Note that we have not imposed boundary conditions at $[l-1,l]\times \partial H_j$. We choose a simple solution. Since the H\"{o}rmander condition is satisfied on $H_j$ there is a density $p^l_j$ for the first equation of (\ref{projectiveHoermanderSystemgenx2rhoH_j}), and we choose
\begin{equation}
u^{\rho ,0,l}_{H_j}(\tau,x)=\int_{H_j}u^{\rho,l-1}(l-1,y)p^l_j(t,x,y)dy.
\end{equation}
Next for $i\in I$ and $B^{\epsilon}_i\cap H=\oslash$ or $u^{\rho ,0,l}_{H_j\cap B^{\epsilon}_i}(l-1,.)$ is not regular on $B^{\epsilon}_i\cap H\neq B^{\epsilon}_i$, then we define $u^{\rho ,0,l}_{H_j}(\tau,x)$ to be a solution of 
\begin{equation}
	\label{projectiveHoermanderSystemgenx2eqrhovec}
	\left\lbrace \begin{array}{ll}
		\frac{\partial u^{\rho,0,l}_{B^{\epsilon}_i}}{\partial \tau}=\rho \tilde{V}_0 u^{\rho ,0,l}_{B^{\epsilon}_i},\\
		\\
		u^{\rho ,0,l}_{B^{\epsilon}_i}(l-1,x)=\chi_{B^{\epsilon}_i}u^{\rho,l-1}(l-1,x).
	\end{array}\right.
\end{equation}
If $i\in I$ and $B^{\epsilon}_i\neq B^{\epsilon}_i\cap H\neq \oslash$, and $u^{\rho ,0,l}_{H_j\cap B^{\epsilon}_i}(l-1,.)=u^{\rho ,l-1}_{H_j\cap B^{\epsilon}_i}(l-1,.)$ is regular on $B^{\epsilon}_i\cap H\neq \oslash$, then we define
\begin{equation}
\label{projectiveHoermanderSystemgenx2eqrhovec2}
	\left\lbrace \begin{array}{ll}
		\frac{\partial u^{\rho,0,l}_{B^{\epsilon}_i}}{\partial \tau}-\rho \sum_{i=1}^n\tilde{V}_0 u^{\rho ,0,l}_{B^{\epsilon}_i},\\
		\\
		=\rho\frac{1}{2}\sum_{i=1}^mV_i^2u^{\rho ,l-1}_{H_j\cap B^{\epsilon}_i}(l-1,.)+\rho V_0u^{\rho ,l-1}_{H_j\cap B^{\epsilon}_i}(l-1,.),\\
		\\
		u^{\rho ,0,l}_{B^{\epsilon}_i}(l-1,x)=\chi_{B^{\epsilon}_i}u^{\rho,l-1}(l-1,x),
	\end{array}\right.
\end{equation}
where $u^{\rho ,l-1}_{H_j\cap B^{\epsilon}_i}(l-1,.)$
 denotes the restriction of $u^{\rho ,l-1}$ to $H_j\cap B^{\epsilon}_i$, and is evaluated at $(l-1,.)$ with time $\tau=l-1$. Note that we have defined a family of functions $u^{\rho ,0,l}_{H_j},~j\in J$ and $u^{\rho ,0,l}_{B^{\epsilon}_i},ĩ\in I$ with some overlap for $i\in K:=\left\lbrace i\in I|B^{\epsilon}_i\neq B^{\epsilon}_i\cap H\neq \oslash\right\rbrace$. We define $u^{\rho,0,l}(\tau,x):=u^{\rho ,0,l}_{H_j},~j\in J$
and define $u^{\rho,0,l}$ on this set such that for all $(\tau,x)\in [l-1,l]\times ({\mathbb R}^n\setminus H)$ we have $u^{\rho,0,l}(\tau,x)=u^{\rho,0,l}_{B^{\epsilon}_i}(\tau,x)$ for some $i$. Next let $u^{\rho,0,l}_{\delta}$ a mollification of $u^{\rho,0,l}$ which converges  with respect to the $L^{\infty}$-norm to the latter function.
 
 Next for $k\geq 0$ we define the local corrections $\delta u^{\rho,k,l}$ at time step $l$. 
 We define
 \begin{equation}
 \delta u^{\rho,0,l}=u^{\rho,0,l}-u^{\rho,l-1}(l-1,.),
 \end{equation}
and 
 \begin{equation}
 \delta u^{\rho,0,l\delta}=u^{\rho,0,l}_{\delta}-u^{\rho,l-1}(l-1,.),
 \end{equation}
 where $u^{\rho,l-1}(l-1,.)$ turns out to be regular inductively.
Next for $j\in J$ on $[l-1,l]\times H_j$ the restriction $\delta u^{\rho,k,l}_{H_j}$ is a solution of
\begin{equation}
	\label{projectiveHoermanderSystemgenx2rhoH_jdelta}
	\begin{array}{ll}
		\frac{\partial \delta u^{\rho ,k,l}_{H_j}}{\partial \tau}-\rho\frac{1}{2}\sum_{i=1}^mV_i^2\delta u^{\rho ,k,l}_{H_j}-\rho V_0 \delta u^{\rho ,k,l}_{H_j}\\
		\\
		=\sum_{j\in J_k}\rho\frac{1}{2}\sum_{i=1}^mV_i^2\delta u^{\rho ,k-1,lj}_{H_j\cap B^{\epsilon}_k}+\rho V_0\delta u^{\rho ,k-1,lj}_{H_j\cap B^{\epsilon}_k}\\
		\\
		+\sum_{q\in I_k}\rho \sum_{i=1}^nV_0\delta u^{\rho ,k-1,lq}_{B^{\epsilon}_q},
	\end{array}
\end{equation} 
where $J_k=\left\lbrace j|H_j\cap B^{\epsilon}_k\right\rbrace$ and
 $I_k=\left\lbrace i|B^{\epsilon}_i\cap B^{\epsilon}_k\right\rbrace\setminus J_k$.
We choose the solution which is determined by the fundamental solution $p^l_j$ of 
\begin{equation}
 \frac{\partial p}{\partial \tau}-\rho\frac{1}{2}\sum_{i=1}^mV_i^2p-\rho V_0p=0
\end{equation}
on $[l-1,l]\times H_j$, i.e., we define
\begin{equation}
\delta u^{\rho ,k,l}_{H_j}(\tau,x):=\int_{l-1}^{\tau}\int_{H_j}\sum_{i\in J_j}L_{H_j\cap B^{\epsilon}_i}
		\delta u^{\rho ,k-1,l}_{B^{\epsilon}_i}(s,y)p^l_j(\tau -s,x,y)dyds,
\end{equation}
where we use the abbreviation $L_{H_j\cap B^{\epsilon}_i}$ which is defined implicitly via (\ref{projectiveHoermanderSystemgenx2rhoH_jdelta}) in an obvious way.
Next we define for $k\in I$, 
\begin{equation}
\label{projectiveHoermanderSystemgenx2eqrhovec3}
	\begin{array}{ll}
		\frac{\partial \delta u^{\rho,k,l}_{B^{\epsilon}_k}}{\partial \tau}-\rho \sum_{i=1}^n\tilde{V}_0\delta u^{\rho ,k,l}_{B^{\epsilon}_k},\\
		\\
		=\sum_{j\in J_k}\rho\frac{1}{2}\sum_{i=1}^mV_i^2\delta u^{\rho ,k-1,lj}_{H_j\cap B^{\epsilon}_k}+\rho V_0\delta u^{\rho ,k-1,lj}_{H_j\cap B^{\epsilon}_k}\\
		\\
		+\sum_{q\in I_k}\rho V_0\delta u^{\rho ,k-1,lq}_{B^{\epsilon}_q}.
	\end{array}
\end{equation}
A solution of the latter equation can be proved as on (\ref{prop}). Then from these families of resticted functions define a function $\delta u^{\rho,k,l}\in L^{\infty}$ analogously as in the definition of $u^{\rho,0,l}$ above, and let $\delta u^{\rho,k,l\delta}$ its mollification with the same mollification operator.
Next we prove that the scheme above converges to a classical solution. For each $l\geq 1$ we define a weighted Sobolev space which is adapted to the Kusuoka-Stroock estimates. We define
\begin{definition}
For any $l\geq 1$ and  consider for $m,p$ the function space
\begin{equation}
\begin{array}{ll}
H^{m,p,l,q}\left([l-1,l]\times {\mathbb R}^n\right):=\\
\\
{\Big \{} f:[l-1,l]\times {\mathbb R}^n \rightarrow 
{\mathbb R}|~(\tau,x)\rightarrow (\tau-(l-1))^p f(\tau,x)\exp\left(-|x|\right) \in H^{m,q}{\Big \}}.
\end{array}
\end{equation}
where $H^{m,q}$ is the Sobolev space with weak spatial derivatives of order $|\alpha|\leq m$ in $L^2$ and weak time derivatives of order $r\leq q$ in $L^2$ (on the domain $[l-1,l]\times {\mathbb R}^n$. We denote the associated norm by $|.|_{m,p,l,q}$
\end{definition}
Starting with step $l=1$ we first observe that $u^{\rho ,0,1}_{H_j}$ and $u^{\rho,0,1}_{B_{\epsilon}}$ are smooth. For given $m$ and $q$ and $l>0$ we choose the time weight $p=q+\max_{j\leq q,|\alpha|\leq m}n_{j,\alpha,\alpha}$. Then there exists $\rho >0$ such that
\begin{equation}
|\delta u^{\rho ,k,1\delta}|_{m,p,1,q}\leq \frac{1}{2}|\delta u^{\rho ,k-1,1\delta}|_{m,p,1,q}.
\end{equation}
For the limit function we have $u^{\rho,1\delta}=u^{\rho ,0,1\delta}+\sum_{k=1}^{\infty}\delta u^{\rho ,k,1\delta}\in H^{m,p,l,q}\left(0,1]\times {\mathbb R}^n\right)$ by construction. Hence for $m\geq k+\frac{1}{2}n$ and for fixed $\tau\in (0,1]$ we have
$u^{\rho,1\delta}(\tau,.)\in H^m\subset C^k$ by the  Sobolev lemma. It follows that $u^{\rho,l\delta}\in C^{1,2}\left((l-1,l]\times {\mathbb R}^n\right)$. For the limit $\delta\downarrow 0$ elementary but cumbersome calculations lead to the concusion that $u^{\rho,1}=\lim_{\delta\downarrow 0}u^{\rho,1\delta}$ is H\"{o}lder continuous with respect to the spatial variables. This can be repeated for the spatial derivatives of $u^{\rho,l\delta}$ leading to the conclusion that $u^{\rho,l\delta}\in C^{1,2}\left((l-1,l]\times {\mathbb R}^n\right)$. For $l\geq 2$ having constructed $u^{\rho,l-1}\in H^{m,p,l,q}\left([l-2,l-1]\times {\mathbb R}^n\right)$  we may  inductively using the transformation of the data above if necessary. 
The argument for spatial derivatives of higher order is similar. Smoothness with respect to time on $(l-1,l]\times {\mathbb R}^n$ for $l\geq 1$ follows from smoothness of the solution function $u^{\rho,l}$ with respect to the spatial variables.  
\end{proof}

\section{Generalization to Lipschitz-continuous coefficients which satisfy a weak H\"{o}rmander condition }
The class of diffusions which satisfy the H\"{o}rmander condition is too narrow for many applications. Especially, dispersion coefficients $\sigma$ (resp. $\sigma\sigma^T$) may be Lipschitz continuous only. Theorem \ref{mainthm} may be applied to classical reduced LIBOR models, but typical extensions with stochastic volatility are not subsumable. However, in our construction above we may apply the partial integration formula of Malliavin calculus for subproblems where a density exists, and this requires essentially that the Malliavin covariant matrix is $L^p$ invertable for all $1\leq p<\infty$. For this reason  
we say that coefficient functions of vector fields satisfy a weak H\"{o}rmander condition on a domain $H\subset {\mathbb R}^n$ if they satisfy the classical H\"{o}rmander condition on a dense set in $H$ and if the Malliavin covariant matrix is $L^p$-invertible for all $1\leq p<\infty$. The latter requirement is quite natural from the point of view of our construction in theorem \ref{mainthm} above, since we assumed the existence of a density on the subdomain $H$, or, to say it differently, the existence of a density on the subdomain $H$ is a consequence of the assumptions of theorem \ref{mainthm} which follows from H\"{o}rmander's theorem.  Note that the typical stochastic volatility models satisfy a stronger form of this weak condition. We may say that coefficient functions of vector fields satisfy a strong form of the weak H\"{o}rmander condition on a domain $H\subset {\mathbb R}^n$ if they satisfy the classical H\"{o}rmander condition almost everywhere in $H$. If we consider weak H\"{o}rmander conditions then we cannot expect to have classical solutions in general (although we shall observe that solution function are of class $C^1$ for a considerable class of problems). Nevertheless, we may adopt the concept of a viscosity solution to the class of equations considered here. We define
\begin{definition}
We say that $u:D:=\left(0,\infty\right)\times {\mathbb R}^n$ with $u(0,x)=h(x)$ is a strong viscosity solution of the ultraparabolic Cauchy problem if 
$u\in C^1\left( \left[0,\infty\right)\times {\mathbb R}^n\right) $ on ${\mathbb R}^n\setminus H$ (where $H$ is the domain where the H\"{o}rmander condition holds), and $u$ is a viscosity solution on $H$ in the traditional sense, i.e. for all $\phi \in P^{2,+}(D)$ (resp. $P^{2,-}(D)$) along with the parabolic upper semijet $P^{2,+}(D)$ (resp. lower semijet $P^{2,-}(D)$) the relations
\begin{equation}
\frac{\partial \phi}{\partial t}-\mbox{Tr}\left( \sigma\sigma^TD^2\phi\right)-\sum_{i=1}^nb_i\frac{\partial \phi}{\partial x_i}\leq (\geq) 0
\end{equation}
are satisfied.
\end{definition}

In order to generalize theorem \ref{mainthm} we need an additional assumption concerning the coefficient functions. We assume 
\begin{theorem}
	\label{mainthmgen}
	Let $1\leq p\leq \infty$ and let $v_{ji}\in C^{\alpha}\left(\left[0,\infty \right)\times {\mathbb R}^n \right)$ for $1\leq i\leq m$ and $1\leq j\leq n$. Assume that the set $H\subseteq {\mathbb R}^n$ where the weak H\"{o}rmander condition holds is a domain and that the parabolic operator degenerates in the complement of the domain or assumption (A) is satisfied. Assume that the initial data function $f:{\mathbb R}^n\rightarrow {\mathbb R}$ satisfies
	\begin{equation}\label{payoff3}
	\begin{array}{ll}
		(i) & ~x\rightarrow f(x)\exp(-C|x|)\in L^p\left({\mathbb R}^n\right) \mbox{ for some $C>0$},\\
		\\
		(ii) & \overline{H}\cup \left\lbrace x|~f~\mbox{is}~C^{\infty}~\mbox{at}~x\right\rbrace ={\mathbb R}^n.
	\end{array}
\end{equation}
Then the Cauchy problem (\ref{projectiveHoermanderSystemgenx2})
	on $[0,\infty)\times {\mathbb R}^n$ has a global strong viscosity solution $u$, where
\begin{equation}
u\in C^{0,1}\left(\left(0,\infty\right) \times {\mathbb R}^n \right). 
\end{equation}
\end{theorem}
\begin{proof}
The plan of the proof is as follows. First we show that there is a sequence of $m$-tuples of vector fields $\left(W^n_1,\cdots W^n_m \right)$ which satisfy the H\"{o}rmander condition on the whole domain of $H$ and which converge uniformly, i.e. locally in the supremum norm, to the $m$-tuple of vector fields $(V_1,\cdots ,V_m)$. Furthermore, using the transformation used in the proof of theorem \ref{mainthm} above we may assume that $f\in L^p\left({\mathbb R}^n\right)$ for $p$ as in the statement of theorem \ref{mainthmgen}. Since the continuous functions with compact support are dense in $L^p\left({\mathbb R}^n\right)$ we may approximate $f$ by a series $f_m\in C^1$ converging in $L^p$. Then we consider Euler-scheme approximations of the diffusion corresponding to each approximating $m$-tuple $\left(W^n_1,\cdots W^n_m \right)$ and initial data $f_m\in C^1$ and use the chain rule and the derivative of the diffusion process with respect to the argument of the expectation value form of the solution function which can be constructed in this approximating situation from the functional series representation given in the proof of theorem \ref{mainthm} above. We compare this representation with an alternative approximations based on the partial integration formula of Malliavin calculus and conclude that the double limit $m\uparrow \infty$ and $n\uparrow \infty$ exists.
We start with the approximation of the vector fields $(V_1,\cdots ,V_m)$.
\begin{lemma}
Let $\left( W_1,\cdots, W_m\right) $ be an of m-tuple of vector fields which have bounded Lipschitz continuous coefficients. Let $H_W$ be the set where the entries $W_i,~1\leq i\leq n$ are smooth and a classical H\"{o}rmander condition holds and let $H^c_W=\overline{H_W}$ be the closure of $H_W$. Then there is a sequence of $m$-tuples of smooth vector fields  $\left( W^n_1,\cdots, W^n_m\right)_{n\in {\mathbb N}} $ which satisfy the H\"{o}rmander condition on $H^c_W$ and converge uniformly to the $m$-tuple $\left( W_1,\cdots, W_m\right) $.
\end{lemma}
\begin{proof}(lemma)
We start with an arbitrary $m$-tuple $\left( W^0_1,\cdots, W^0_m\right) $
 of smooth vector fields $W^0_j,~1\leq j\leq m$ which satisfies the H\"{o}rmander condition on a domain $U$ such that $H^c_W\subseteq U$ (equality may occur if $H^c_W$ is the whole space of ${\mathbb R}^n$). At each stage $q$ of our construction we construct an $m+1$-tuple $\left(W^q_0, W^q_1,\cdots, W^q_m\right) $ which satisfies the H\"{o}rmander condition at a set of points chosen from $$J^0_q=\left\lbrace \frac{r}{2^{q}}+I_{\frac{1}{2^{q+1}}},~r\in {\mathbb Z}^n\right\rbrace , $$ 
 where $\frac{r}{2^{q}}+I_{\frac{1}{2^{q+1}}}$ is the cube with mid point $\frac{r}{2^{q}}$ and edge size $\frac{1}{2^{q+1}}$.
  Note that $r=(r_1,\cdots ,r_n)^T$. 
  Let $J_q:=\left\lbrace U_k|U_k\in J^0_q \mbox{ and } U_k\subset H \right\rbrace=\left\lbrace U_k| k\in K_q\right\rbrace$ 
 for some index set $K_q$. The choice is not arbitrary. 
 Since the H\"{o}rmander condition holds on a dense subset of 
 $H$ we can choose a tuple $(x^q_j)_{j\in K_q}$ (of cardinality $|K_q|$) 
 such that $x^q_j\in U_j\in J_q$ 
 for all $j\in K_q$, and such that for each $j\in K_q$
 \begin{equation}\label{condmid}
 |x^q_j-r|=\max{i\in\left\lbrace 1,\cdots, n\right\rbrace}|x^q_{ji}-r_i|\leq \frac{1}{2^{q+2}},
 \end{equation}
and by induction we may assume that for all $1\leq p\leq q-1$
\begin{equation}\label{condstage}
 |x^q_j-x^p_{j}|=\max{i\in\left\lbrace 1,\cdots, n\right\rbrace}|x^q_{ji}-x^m_{ji}|\geq \frac{1}{2^{q+1}}
 \end{equation}
Note that the $x^q_j$ have distance greater $\frac{1}{2^{q+1}}$.   
The main idea is to add at each stage a function which does not alter the H\"{o}rmander condition of the points chosen at the previous stages of construction and such that the H\"{o}rmander condition is satisfied at the additional points chosen at the present stage.    
Assume that $\left(W^q_0, W^q_1,\cdots, W^q_m\right) $ has been constructed. Then we define in a first substage of stage $q+1$
\begin{equation}
 W^{q+1,0}_j=W^{q}_j+\sum_{k\in J_q}c^{k,q}_j\phi_{\frac{1}{ 2^{q+2}}}(x-x_j),
\end{equation}
for all $0\leq j\leq m$ where for each $\epsilon >0$
\begin{equation}
\phi_{\epsilon}(x)=e\cdot\exp\left(-\frac{\epsilon^2}{\epsilon^2-|x|^2} \right), 
\end{equation}
and where the constants $c^{k,q}_j=\left(c^{k,q}_{j1},\cdots ,c^{k,q}_{jn}\right) $ are determined by the relation
\begin{equation}
W^{q}_j(x_j)-W_j(x_j)=c^{k,q}_j\mbox{ for all }j\in J_n.
\end{equation}
In a second substage then we first determine the next set $(x^{q+1}_j)_{j\in K_{q+1}}$ with the properties in (\ref{condmid}) and in (\ref{condstage}) for $q+1$ instead of $q$. Then we know that
the distance of the points in the latter tuple greater than $\frac{1}{2^{q+2}}$ and that their distance to the points constructed at earlier stages is also greater than $\frac{1}{2^{q+2}}$. Then for each $j\in K_{q+1}$ let $B_{\frac{1}{2^{q+4}}}(x^{q+1}_j)$ be the ball of radius $\frac{1}{2^{q+4}}$ around the midpoint $x^{q+1}_j$. For each $j\in K_{q+1}$ and $1\leq i\leq m$ define a function $\psi^j_i\in C^{\infty}_c\left(B_{\frac{1}{2^{q+3}}}(x^{q+1}_j) \right)$, i.e. with support in $B_{\frac{1}{2^{q+3}}}(x^{q+1}_j)$ which equals $W_i-W^{q+1,0}_i$ on $B_{\frac{1}{2^{q+4}}}(x^{q+1}_j)$. Then define $W^{q+1}_i=W^{q+1,0}_i-\sum_{j\in K_{q+1}}\psi^j_i$ for all $1\leq i\leq m$.
This construction can be repeated arbitrarily often and has a limit $W$ with respect to the supremum norm.\end{proof}

Next we verify that the Kusuoka-Strook estimates of first order (in the probabilistic form) are stable in the situation of the preceding lemma.
This follows from the stability of the Malliavin partial integration formula. 
First we take a sequence of data $f_p\in C^1\cap L^p$ such that $\lim_{p\uparrow \infty} |f_p-f|_{L^p}=0$ and a sequence of $m$-tuple $(W^q)_{q\in {\mathbb N}}$ of $m$-tuples of vector fields such that $\lim_{q\uparrow \infty}|W^q_{ji}-v_{ji}|_0=0$ for all $j=1,\cdots m$ and $1\leq i\leq n$. Let $X^q$ the diffusion process with drift coefficients $b^q_i$ and dispersion coefficients $\sigma^{q}_{ij}$ corresponding to the $m$-tuple of vector fields $W^q$ such that $(t,x)\rightarrow E^x\left(f_m\left(X^q_t\right)\right)$ solves the Cauchy problem
 \begin{equation}
	\label{projectiveHoermanderSystemgenx22}
	\left\lbrace \begin{array}{ll}
		\frac{\partial u^q}{\partial t}
		=\frac{1}{2}\sum_{i=1}^m \left( W^q_i\right) ^2u^q
		+W^q_0u^q\\
		\\
		u^q(0,x)=f_m(x).
	\end{array}\right.
\end{equation}
We denote the limit process by $X$ with drift coefficients $b_i$ and dispersion coefficients $\sigma_{ij}$ corresponding to the $m$-tuple of vector fields $V.$  
We know from (\ref{mainthm}) that the function $(t,x)\rightarrow E^x\left(f_m\left(X^q_t\right)\right)$ is smooth for each $m,q\in{\mathbb N}$.
 Then we have
 \begin{equation}\label{exp2}
 \begin{array}{ll}
\frac{\partial}{\partial x_i}E^x\left( f_m\left(X^q_t \right)\right)=
E^x\left( \left( \frac{\partial}{\partial x_i}f_m\right) \left(X^q_t \right)Y^{q}_{ij}\right)
\end{array}
\end{equation}
Then we may differentiate with respect to $x$ and compute 
with the matrix-valued process $\frac{\partial}{\partial x_j}X^q=Y^{q}_{ij}$ to be
\begin{equation}
\begin{array}{ll}
\frac{\partial}{\partial x_j} X^q=Y^{q}_{ij}
=\delta_{ij}+\sum_{k=1}^n\int_0^t
\left( \frac{\partial}{\partial x_k}b^q_i(X^q_s)\right) 
Y^{k}_j(s)ds\\
\\+\sum_{l,k=1}^n\int_0^t\left( \frac{\partial}{\partial x_k}\sigma_{il}(X^q_s)\right) Y^k_j(s)dW^l_s.
\end{array}
\end{equation}
We approximate the processes $X^q$ and $Y^q$ by Euler schemes $X^{q,e_{\Delta}}$ and $Y^{q,e_{\Delta}}$ of time-step size $\Delta$. According to Rademacher's theorem the set $S$ of arguments where one of the Lipschitz continuous dispersion functions $\sigma_{ij}$ or one of the drift coefficient functions $b_i$ are non-differentiable is of measure zero. For $x\in S$ we replace the derivatives of $\sigma_{ij}$ and of $b_i$ in (\ref{exp2}) by difference quotients with difference $h>0$ and the the processes $X^q$ and $Y^q$ by Euler-scheme approximations $X^{q,e_{\Delta}}$ and $Y^{q,e_{\Delta}}$. As the difference quotients of  $\sigma_{ij}$ and of $b_i$ are uniformly bounded by a Lipschitz constant $M$ (independent of the difference $h$). Since the Euler schemes have values in $S$ of probability measure zero the limit with  $q\uparrow \infty$ and $\Delta, h\downarrow 0$ exists for each $m$. 
Next for $f_m\in C^1$ and the weight $W=\det\left(\gamma_{ij}\right)^2 $ along with the Malliavin covariance matrix $\gamma_{ij}$ we may rewrite the expectation with a changed measure such that the partial integration formula
\begin{equation}
 E\left(\frac{\partial}{\partial x_i}f_m(X)W\right)=E\left(f_m(X)H^i\left(X,W\right)\right)  
\end{equation}
 with 
 \begin{equation}\label{1M}
 H^i(X,W)=-\sum_{j=1}^nW\gamma^{-1}_XL(X^j)+\left\langle D X^j,D\left(W\gamma^{-1,ij}_X \right) \right\rangle 
 \end{equation}
is valid. Here we use the fact the covariance matrix is in $L^p$ and $L^p$-invertible for $p\geq 1$. Taking the limit $m\uparrow \infty$ leads to the result.
\end{proof}
\begin{remark}
It seems that with the methods of Malliavin calculus it is difficult to establish more regularity for diffusions which satisfy a weak H\"{o}rmander condition.
The difficulty is to establish the existence of higher order weights $H^{\alpha}\left(X,W\right)$ which appear in the Malliavin integration by parts formula. For higher derivatives this is
\begin{equation}\label{hM}
E\left(D^{\alpha}f\left(X\right),W\right)=E\left(f(X)H^{\alpha}\left(X,W\right)  \right)  
\end{equation}
with recursively define weights $H^{\alpha}$ such that
\begin{equation}
H^{(\alpha,\alpha_{m+1})}\left(X,W\right)=H^{\alpha_{m+1}}\left( X, H^{\alpha}\left(X,W \right) \right)  
\end{equation}
along with $(\alpha,\alpha_{m+1})=(\alpha_1,\cdots,\alpha_m,\alpha_{m+1})$.
\end{remark}
\section{Weighted Monte-Carlo algorithms related to a class of ultraparabolic diffusions}

Next we describe the probabilistic scheme related to the class of ultraparabolic equations described in this paper. Furthermore we compare the scheme with other probabilistic schemes proposed in the literature, especially the schemes in \cite{FriesKampenProxy2005}, \cite{FJ}, and \cite{KKS}.
Expectation values and their sensitivities of diffusion processes $X$ starting at $X_0=x_0$ of the form
\begin{equation}\label{origdiff}
X_t=x_0+\int_0^t\mu(s,X_s)ds+\int_0^t\sigma(s,X_s)dW_s
\end{equation}
are usually computed via discretizations of the process. For example in an Euler scheme with time discretization $t_0=0<t_1<\cdots <t_N=:T$ the expectation value
\begin{equation}\label{MCvalue}
E^{x_0}\left(f\left(X_T\right)  \right) 
\end{equation}
may be approximated by a number 
\begin{equation}\label{MCsum}
\frac{1}{N}\sum_{j=1}^Nf\left(X^e_T(\omega_j)\right),
\end{equation}
where $N$ is the number of paths denoted by $\omega_j$ and $X^e_T(\omega_j)$ refers to the evaluation of the Euler scheme $X^e$ at time $T$ and for the path $\omega_j$. Similar for sensitivities, which are derivatives of the expectation value function with respect to an argument (e.g. a component of $x_0$) or an parameter. Starting at $t_0$ with the vector $x_0$ at each time step $t_i$ a vector of random numbers with distribution of the random variable $\Delta W_{t_i}=W_{t_{i+1}}-W_{t_i}$ is drawn. The vector of random numbers drawn at time step $i$ for the $j$th path may denoted by $\Delta W_{t_i}(\omega_j)$ and determines the evaluation of the random variable $X^e_{t_{i+1}}(\omega_j)$, i.e., the value of the Euler scheme at the $i+1$ time step and at the $j$th path via
\begin{equation}
X^e_{t_{i+1}}(\omega_j)=X_{t_i}(\omega_j)+\mu\left( t_i,X^e_{t_i}(\omega_j)\right)\Delta t_i+\sigma\left( t_i,X^e_{t_i}(\omega_j)\right)\Delta W_{t_i}(\omega_j).
\end{equation}
Sometimes certain paths are more likely to contribute to the value of (\ref{MCsum}) than others, and in this case it may be an advantage to approximate the value (\ref{MCvalue}) by a weighted sum
\begin{equation}\label{MCsumweighted}
\frac{1}{N}\sum_{j=1}^Nw_jf\left(X^e_T(\omega_j)\right).
\end{equation} 
where each path is 'weighted' with the real number $w_j$. Such schemes are called weighted Monte-Carlo algorithms. The choice of the weight is an art in itself. For example, derivatives of option prices with respect to underlying, so-called $\Delta$s, may lead to the problem of simulating highly peaked distributions and to the choice of very special weights in order to make the computation stable (cf. \cite{KKS}). Now assume that the diffusion has a density and that there is a nice approximation of that density which can be used in order to evaluate a probabilistic scheme for a diffusion in one time step only. For example in \cite{KKS} weights for the computation of $\Delta$s for classical interest rate options with a maturity of $10$ years in one time step are computed in the framework of the classical LIBOR market model via WKB-expansions of densities $p$ of a diffusion $X$ evaluated at each path, i.e. WKB-approximations of the numbers
\begin{equation}
(t,x_0)\rightarrow p\left( t_0,x_0,T,X_T(\omega_j)\right).
\end{equation}
The WKB-approximations $p_{\mbox{WKB}}$ of the density $p$ define a target scheme $X^*$ which approximates the original diffusion $X$.
 The evaluation involves the idea of a full proxy scheme (cf. \cite{FriesKampenProxy2005}), where the random variable is evaluated with respect to an easily computable prior $X^0$ such that
 \begin{equation}
 E\left(f(X^*_T)|_{{\cal F}_{t_0}}\right)=E\left(f(X^0_T)W_T|_{{\cal F}_{t_0}}\right)
 \end{equation}
 along with
\begin{equation}
W_T=\frac{p_{\mbox{WKB}}\left( t_0,x_0,T,X^0_T(\omega_j)\right)}{\phi^e(t_0,x_0,T,X^0_{T}(\omega_j))},
\end{equation}
and where $\phi^e$ refers to the density of the Euler scheme (which is just a linear transformation of the Gaussian). We note that this choice of the weights is called the 'naive scheme' in (\cite{KKS}) for reason explained there, but in many cases it works. Similar for multiple time steps where weights are represented by products of weights computed at each time step. Now experience shows that the full factor LIBOR market model is not needed in order to match requirements of calibration and may be too cumbersome in order to do actual computations (maybe involving $>20$ underlyings). For this reason reduced LIBOR market models are constructed. 
Reduced LIBOR market models have no densities in general, and for this reason the scheme presented in \cite{KKS} has to be altered. If the Malliavin covariance matrix of the diffusion is $L^p$ invertible for all $p\geq 1$, then probabilistic schemes based on Malliavin weights may be considered (cf.\cite{EFT}). However, this is not true for the reduced LIBOR market model in general (with and without stochastic volatility). We may still use a density on the subspace where it exists, but we have to ensure that the target scheme does not attribute any mass whenever the prior has none. Note that in the case of stochastic volatility a simple Euler scheme for the prior may lead to the phenomenon that the prior density is supports only a subspace which may be of lower dimension than the subspace where a H\"{o}rmander condition holds. Note that the volatility matrix may be only of rank $k<n$ while the subspace where the H\"{o}rmander condition holds may be of rank $d>k$. From the perspective of the practical incentives of this paper our interest is the case where $d< n$. If $k$ denotes the rank of the diffusion matrix $\sigma\sigma^T$ we may have the situation $k<d<n$. For example this is true for $k=1$ $d=2$ and $n=3$ for the diffusion related to the equation
Consider the equation
\begin{equation}\label{examplehd}
\frac{\partial u}{\partial t}-\lambda_2\frac{\partial^2 u}{\partial x_2^2}+x_2\mu_1\frac{\partial u}{\partial x_1}+\mu_2\frac{\partial u}{\partial x_3}=0
\end{equation}
for some constants $\lambda_2,\mu_1,\mu_2>0$. 
\begin{equation}
H:=\mbox{span}\left\lbrace (0,\lambda_2,0)^T,(\lambda_2\mu_1,0,0\right\rbrace={\mathbb R}^2=H_x 
\end{equation}
independently of the argument $x$. Note that we have a density on the whole space for the reduced equation
\begin{equation}
\frac{\partial u}{\partial t}-\lambda_2\frac{\partial^2 u}{\partial x_2^2}+x_2\mu_1\frac{\partial u}{\partial x_1}=0,
\end{equation}
because the coefficients are of linear growth. In example (\ref{examplehd}) it is only the drift which contributes to the difference of the H\"{o}rmander dimension and the rank of the diffusion matrix. Note that it is possible that the dimension of $H_x$, i.e. the H\"{o}rmander space at some specific point $x\in {\mathbb R}^n$, is larger than the rank of the diffusion matrix, where the drift and all its derivatives are zero at this point. However, this set of degeneracies  should be $L^p$-invertible (for $p\geq 1$) in order to construct a (regular) density. Examples may be constructed with a more complicated interplay of spatial dependence of the diffusion matrix $\sigma\sigma^T$ and spatial dependence of the drift where significant differences $d-k\geq 2$ may occur in regions which are not of Lebesgue measure zero. For this reason it is a desideratum of present research to have numerical constructions of target densities which are supported on the whole of the H\"{o}rmander subspace - even for small time steps. It is clear that the WKB-expansion does not satisfy this strong requirement in general. Even in example \ref{examplehd} we may compute the WKB-expansion formally (with explicit solutions for the WKB-coefficients), but it is obvious that this leads to poor numerical target density schemes. 
partial proxy scheme in \cite{FJ} and  \cite{FriesLectureNotes2007}) can be extended.
For the usual proxy scheme the prior should be supported on $H$. Note that in this paper the domain $H$ is time-invariant. However, this is for the sake of simplified notation since the extension to the time-dependent case causes no further problems. 
Anyway, the theoretical considerations of this paper lead to an algorithmic frame (based on the constructive scheme above) which may be used in order to improve the existing schemes by adding correction terms which may be computed on the basis of Malliavin calculus. Next we describe this scheme and mention some specifications of this frame. Note that the scheme is compatible with the proxy scheme and the partial proxy scheme.
\begin{itemize}
 \item[i)] Compute the domain $H$, where the (weak) H\"{o}rmander condition holds. This domain can be computed from the coefficient functions without reference to the data. In most cases of practical interest this domain can be computed easily  as is the case for reduced classical LIBOR market models. Note that for reduced LIBOR market models the difference of the dimension of the H\"{o}rmander space and the rank of the diffusion matrix can be large. In case of time dependence of the coefficients $H$ is replaced by a time-parameterized family $\left(H_t\right)_{t\geq 0}$ where each $H_t$ is computed as in the time-homogeneous case. For the sake of simplicity we describe the following algorithm for time-invariant $H$, where an extension to the time-dependent case is rather trivial. 
 \item[ii)] Consider a time discretization $0=t_0<t_1<t_2<\cdots t_N=T$ and approximate the diffusion (\ref{origdiff}) by a prior scheme $X^p$ with
\begin{equation}\label{xp}
X^p_{t_{i+1}}(\omega_j)=X^p_{t_i}(\omega_j)+\mu\left( X^p_{t_i}(\omega_j)\right)\Delta t_i+\sigma^p\left(X^p_{t_i}(\omega_j)\right)\Delta W_{t_i}(\omega_j)
\end{equation}
for $\omega_j,~j\in J$ with $J$ some finite index set. Ensure that 
\begin{equation}
\sigma^p\equiv 0~\mbox{ if }~x\in {\mathbb R}^n\setminus H.
\end{equation}
The latter condition refers to the model assumption of a reduced diffusion (the class of ultraparabolic equations considered here).
We shall discuss below how the matrix $\sigma^p$ may be chosen if there are good approximations of densities such that a proxy scheme may be used. At each time step $t_i$ and for all $j\in J$ consider four possibilities of staying in the H\"{o}rmander domain (case $H^{ii}$), staying in the complement of the H\"{o}rmander domain (case $H^{00}$), going out of the H\"{o}rmander domain (case $H^{i0}$), and going into the H\"{o}rmander domain (case $H^{0i}$) at time step $t_{i+1}$:
\begin{itemize}
\item[$H^{00}$:] $X^p_{t_{i}}(\omega_j)\in {\mathbb R}^n\setminus H$ and $X^p_{t_{i+1}}(\omega_j)\in {\mathbb R}^n\setminus H$,
\item[$H^{ii}$:] $X^p_{t_{i}}(\omega_j)\in  H$ and $X^p_{t_{i+1}}(\omega_j)\in H$,
\item[$H^{0i}$:] $X^p_{t_{i}}(\omega_j)\in {\mathbb R}^n\setminus H$ and $X^p_{t_{i+1}}(\omega_j)\in H$,
\item [$H^{i0}$:] $X^p_{t_{i}}(\omega_j)\in  H$ and $X^p_{t_{i+1}}(\omega_j)\in {\mathbb R}^n\setminus H$.
\end{itemize}
\item[iii)] Next we compute the weights at each time step. Assume that the weights have been computed up to time step $t_i$, i.e., the weights $W_{t_m},~0\leq m\leq i$ are known.
\begin{itemize}
\item[$W^{00}$:] If $H^{00}$ holds, i.e., if $X^p_{t_{i}}(\omega_j)\in {\mathbb R}^n\setminus H$ and $X^p_{t_{i+1}}(\omega_j)\in {\mathbb R}^n\setminus H$, then a deterministic vector field equation has to be simulated for the path $\omega_j$ at time step $t_i$. Accordingly, we choose an Euler step in this case, i.e.,
\begin{equation}
\begin{array}{ll}
X^e_{t_{i+1}}(\omega_j)-X^p_{t_i}(\omega_j)=\mu\left( t_i,X^p_{t_i}(\omega_j)\right)\Delta t_i+\sigma\left( t_i,X^p_{t_i}(\omega_j)\right)\Delta W_{t_i}(\omega_j)\\
\\
=\mu\left( t_i,X^p_{t_i}(\omega_j)\right)\Delta t_i
\end{array}
\end{equation}
along with the original volatility matrix $\sigma$. Accordingly the weight is
\begin{equation}
W_{t_{i+1}}\equiv 1.
\end{equation}

\item[$W^{ii}$:] If $H^{ii}$ holds, i.e. if $X^p_{t_{i}}(\omega_j)\in  H$ and $X^p_{t_{i+1}}(\omega_j)\in H$, then we compute a Malliavin weight or a proxy weight. For example if a higher order WKB-expansion on $H$ is available and a good approximation, then we may choose
\begin{equation}
W_{t_{i+1}}(\omega_j)|_{F_{t_i}}=\frac{p_{\mbox{WKB}}\left( t_i,X^p_{t_i}(\omega_j),t_{i+1},X^p_{t_{i+1}}(\omega_j)\right)}{\phi^p(t_i,X^p_{t_i}(\omega_j),t_{i+1},X^p_{t_{i+1}}(\omega_j))}, 
\end{equation}
where $\phi^p$ denotes the density of the prior scheme.
  
\item[$W^{0i}$:] If $H^{0i}$ holds, i.e. $X^p_{t_{i}}(\omega_j)\in {\mathbb R}^n\setminus H$ and $X^p_{t_{i+1}}(\omega_j)\in H$, then there are several possibilities, depending on the geometry and the efficiency of the computability of the boundary of  $H$. We use 
\begin{definition}\label{mcdef}
The boundary of the H\"{o}rmander space is topologically simple with respect to a time discretization step $i+1$ of a scheme $X^p$ if
\begin{equation}
\left\lbrace \lambda X^p_{t_{i+1}}(\omega_j)+(1-\lambda) X^p_{t_{i}}(\omega_j)|\lambda\in [0,1]\right\rbrace \cap \partial H=\left\lbrace z^p_i \right\rbrace 
\end{equation}
for a singleton $\left\lbrace z^p_i \right\rbrace$. Here $\partial H$ denotes the boundary of $H$.
\end{definition}
If the boundary of $H$ is topologically simple with respect to to a time discretization step $i+1$ of a scheme $X^p$ and easily computable, then we can proceed as follows: we compute $t^H_{i}$ such that  $X^p_{t^H_{i}}(\omega_j)=z^p_i$. Then for the the first substep from $t_i$ to $t^H_i$ we are in the situation $W^{00}$ and proceed accordingly. For the second substep from $t^H_i$ to $t_{i+1}$ we define
\begin{equation}
W_{t_{i+1}}(\omega_j)|_{F_{t^H_i}}=\frac{p_{T}\left( t_i,X^p_{t_i}(\omega_j),t_{i+1},X^p_{t_{i+1}}(\omega_j)\right)}{\phi^p(t^H_i,X^p_{t^H_i}(\omega_j),t_{i+1},X^p_{t_{i+1}}(\omega_j))},
\end{equation}
where $p_T$ is a numerical approximation of the target density. 
There is a more general possibility which may be applied especially if the boundary of $H$ is not easily computable. In this case we consider two subcases. If $X^p_{t_{i-1}}(\omega_j)\in {\mathbb R}^n\setminus H$ then we cannot simulate a density from the preceding two time steps, and, hence, define 
\begin{equation}
X^e_{t_{i+1}}(\omega_j)-X^p_{t_i}(\omega_j)=\mu\left( t_i,X^p_{t_i}(\omega_j)\right)\Delta t_i+\sigma\left( t_i,X^p_{t_i}(\omega_j)\right)\Delta W_{t_i}(\omega_j),
\end{equation}
with the weight.
\begin{equation}
W_{t_{i+1}}\equiv 1.
\end{equation}
Otherwise, if $X^p_{t_{i-1}}(\omega_j)\in  H$, then we we compute a Malliavin weight or a proxy weight as a transition from $t_{i-1}$ two $t_i$. For example if a higher order WKB-expansion on $H$ is available and a good approximation, then we may choose
\begin{equation}
W_{t_{i+1}}(\omega_j)|_{F_{t_{i-1}}}=\frac{p_{\mbox{WKB}}\left( t_{i-1},X^p_{t_{i-1}}(\omega_j),t_{i+1},X^p_{t_{i+1}}(\omega_j)\right)}{\phi^e(t_{i-1},X^p_{t_{i-1}}(\omega_j),t_{i+1},X^p_{t_{i+1}}(\omega_j))}. 
\end{equation}
\item [$W^{i0}$:] If $H^{i0}$ holds, i.e., if $X^p_{t_{i}}(\omega_j)\in  H$ and $X^p_{t_{i+1}}(\omega_j)\in {\mathbb R}^n\setminus H$, then we proceed as follows.
If the boundary of $H$ is topologically simple with respect to to a time discretization step $i+1$ of a scheme $X^p$ and easily computable, then we can proceed as follows: we compute $t^H_{i}$ such that  $X^p_{t^H_{i}}(\omega_j)=z^p_i$ as in definition (\ref{mcdef}) above. Then for the the first substep from $t_i$ to $t^H_i$ we are in the situation $W^{ii}$ and proceed accordingly. For the second substep from $t^H_i$ to $t_{i+1}$ we are in the situation $W{00}$ and proceed accordingly. 
Otherwise we choose $W_{t_{i+1}}\equiv 1$
\end{itemize} 
\item[iv)] Our first approximation to for the value function is 
 Compute 
 \begin{equation}
 E\left(f(X^*_{T})|_{{\cal F}_{t_0}}\right)
 =E\left(f(X^0_T)W_T|_{{\cal F}_{t_0}}\right)
 \end{equation}
 along with
\begin{equation}
W_T|_{F_{t_0}}=\Pi_{k=0}^{N-1}W_{t_{k+1}}.
\end{equation}
Similarly for sensitivities. For example if $X^*_{0}=x$, then 
\begin{equation}
 \frac{\partial}{\partial x_m}E\left(f(X^*_{T})|_{{\cal F}_{t_0}}\right)
 =\frac{\partial}{\partial x_m}E\left(f(X^0_T)W_T|_{{\cal F}_{t_0}}\right)
 \end{equation}
may be computed by explicit derivatives or finite differences of the weight $W_T$.
Finally, the correction terms of the weighted MC-scheme are computed according to our construction of a regular solution above. We set the correction $C_{t_0}=C_0\equiv 0$ and describe the formula for the correction scheme $C_{t_i}$ in the value formula 
\begin{equation}
 u(t,x)=E\left(f(X^0_T)W_T|_{{\cal F}_{t_0}}\right)+\sum_{i=1}^NE(\Delta C_{t_i}(X))
\end{equation}
recursively. Here $\Delta C_{t_i}(\omega_j):=C_{t_{i+1}}(\omega_j)-C_{t_i}(\omega_j)$ define the correction increments at each time step $t_i$. Similarly for sensitivities. 
\begin{itemize}
\item[$C^{00}$:] If $H^{00}$ holds, i.e., if $X^p_{t_{i}}(\omega_j)\in {\mathbb R}^n\setminus H$ and $X^p_{t_{i+1}}(\omega_j)\in {\mathbb R}^n\setminus H$, then 
\begin{equation}
\Delta C_{t_i}(\omega_j)\equiv 0.
\end{equation}
\item[$C^{ii}$:] If $H^{ii}$ holds, i.e. if $X^p_{t_{i}}(\omega_j)\in  H$ and $X^p_{t_{i+1}}(\omega_j)\in H$, then 
\begin{equation}
\Delta C_{t_{i}}(\omega_j)\equiv 0.
\end{equation}
\item[$C^{0i}$:] If $H^{0i}$ holds, i.e. $X^p_{t_{i}}(\omega_j)\in {\mathbb R}^n\setminus H$ and $X^p_{t_{i+1}}(\omega_j)\in H$, then we consider two subcases. If $X^p_{t_{i-1}}(\omega_j)\in {\mathbb R}^n\setminus H$ then we cannot simulate a density from the preceding two time steps, and, hence, define 
\begin{equation}
\Delta C_{t_i}\equiv 0.
\end{equation}
Otherwise, if $X^p_{t_{i-1}(\omega_j}\in  H$ then we compute $\Delta C_{t_{i}}(\omega_j)$ via a Malliavin weight, i.e.,
\begin{equation}
\begin{array}{ll}
\Delta C_{t_{i+1}}(\omega_j)=\\
\\
\sum_{1\leq j,k\leq n}\left( f(X^p_{t_{i+1}})H^{jk}\left(X^p_{t_{i+1}},W^{jk}_{t_{i+1}}\right)-f(X^p_{t_{i-1}})H^{jk}\left(X^p_{t_{i-1}},W^{jk}_{t_{i-1}}\right)\right)\\
\\
-\sum_{j=1}^n\left( f(X^p_{t_{i+1}})H^{j}\left(X^p_{t_{i+1}},W^{j}_{t_{i+1}}\right)-f(X^p_{t_{i-1}})H^{j}\left(X^p_{t_{i-1}},W^{j}_{t_{i-1}}\right)\right),
\end{array}
\end{equation}
where for $m\in \left\lbrace i-1,i+1\right\rbrace$ 
\begin{equation}
W^j_{t_m}:=-\sum_{i=1}^nb_i(X^p_{t_m}),
\end{equation}
and
\begin{equation}
W^{jk}_{t_m}:=\sum_{j,k=1}^n\frac{1}{2}\left( \sigma\sigma^T\right)_{jk}(X^p_{t_m}).
\end{equation}
Furthermore, $H^j$ and $H^{jk}$ are the first order and second order Malliavin weights defined in (\ref{1M}) and (\ref{hM}) above. Here $H^{jk}$ refers to the Malliavin weight $H^{\alpha}$ along with multi-index $\alpha=(\alpha_1,\cdots,\alpha_n)$, where $\alpha_j=\alpha_k=1$ and $\alpha_m=0$ for $m\neq 0$.
\item [$C^{i0}$:] If $H^{i0}$ holds, i.e., if $X^p_{t_{i}} (\omega_j)\in  H$ and $X^p_{t_{i+1}}(\omega_j)\in {\mathbb R}^n\setminus H$, then we choose $\Delta C_{t_{i+1}}\equiv 0$. We could consider some subclass and improve the algorithm a bit at this point.
\end{itemize} 
\item[iv)] The computation scheme then is
 \begin{equation}
 E\left(f(X^*_{T})|_{{\cal F}_{t_0}}\right)
 =E\left(f(X^0_T)W_T|_{{\cal F}_{t_0}}\right)+\sum_{i=1}^NE(\Delta C_{t_i})
 \end{equation}
 along with
\begin{equation}
W_T|_{F_{t_0}}=\Pi_{k=0}^{N-1}W_{t_{k+1}},
\end{equation}
and the other weights $W^j_{t_i}$ and $W^{jk}_{t_i}$ implicit in $\Delta C_{t_{i}}$ and defined above.
Similarly for sensitivities.
\end{itemize}

\footnotetext[1]{DZ BANK, Platz der Republik, 60265 Frankfurt, Germany {\tt Christian Fries <email@christian-fries.de>}}
\footnotetext[2]{Weierstrass Institute for Applied Analysis and Stochastics, Mohrenstr. 39, D-10117 Berlin, Germany. {\tt kampen@wias-berlin.de}}

\end{document}